\documentclass[12pt]{article}
\usepackage{setspace}
\usepackage{a4}
\usepackage{amssymb,amsmath,amsthm,latexsym}
\usepackage{amsfonts}
\usepackage{amsfonts}
\usepackage{graphicx}
\usepackage{textcomp}
\usepackage{tikz-cd}
\usetikzlibrary{cd}

\usepackage{cite}
\newtheorem{theorem}{Theorem}[section]

\newtheorem{corollary}[theorem] {Corollary}
\newtheorem{definition}[theorem]{Definition}

\newtheorem{proposition}[theorem]{Proposition}
\newtheorem{remark}[theorem]{Remark}
\title{This is the title}
\usepackage{amssymb}
\usepackage{amssymb}
\usepackage{amssymb}
\usepackage{amssymb}
\usepackage{amsmath}
\usepackage{tikz}
\usepackage{hyperref}
\usepackage{enumerate}
\usepackage{mathtools}
\usepackage{amsmath}
\usepackage{tikz}

\begin{document}
\begin{center}
{\bf{MULTIPLIERS FOR  LIPSCHITZ p-BESSEL SEQUENCES IN METRIC  SPACES}}\\
K. MAHESH KRISHNA AND P. SAM JOHNSON  \\
Department of Mathematical and Computational Sciences\\ 
National Institute of Technology Karnataka (NITK), Surathkal\\
Mangaluru 575 025, India  \\
Emails: kmaheshak@gmail.com, kmaheshakma16f02@nitk.edu.in\\
nitksam@gmail.com, sam@nitk.edu.in

Date: \today
\end{center}

\hrule
\vspace{0.5cm}
\textbf{Abstract}: The notion of multipliers in Hilbert space was introduced by
Schatten in 1960 using orthonormal sequences and  was generalized by Balazs in
2007 using Bessel sequences. This was  extended to Banach spaces by Rahimi and
Balazs in 2010 using p-Bessel sequences. In this paper, we further extend this
by considering Lipschitz functions. On the way we define frames for metric
spaces which extends the notion of frames and Bessel sequences for Banach spaces. We show that when the symbol sequence converges to zero, the
multiplier is a Lipschitz compact operator. We study how the variation of
parameters in the multiplier effects the properties of multiplier.



\section{Introduction}

Let $\{\lambda_n\}_n \in \ell^\infty(\mathbb{N})$ and  $\{x_n\}_n$, $\{y_n\}_n$ be
sequences in a Hilbert space $\mathcal{H}$. For $x,y \in \mathcal{H}$, the
operator $x\otimes \overline{y}$ is defined by  	$x\otimes
\overline{y}:\mathcal{H}\ni h \mapsto \langle h, y\rangle x\in \mathcal{H}$.

The study of operators of the form 
 \begin{align}\label{FIRST EQUATION}
 	\sum_{n=1}^{\infty}\lambda_n (x_n\otimes \overline{y_n})
 \end{align}
 began with Schatten \cite{SCHATTEN}, in connection with the study of compact
operators. Schatten studied the operator in (\ref{FIRST EQUATION}) whenever
$\{x_n\}_n$, $\{y_n\}_n$ are orthonormal sequences in a Hilbert space
$\mathcal{H}$. Later,  operators in (\ref{FIRST EQUATION}) are studied mainly in
connection with Gabor analysis \cite{FEICHTINGERNOWAK, BENEDETTOPFANDER,
DORFLERTORRESANI, GIBSONLAMOUREUXMARGRAVE, CORDEROGROCHENIG, SKRETTINGLAND}. This was generalized by Balazs
\cite{BALAZSBASIC} who replaced orthonormal sequences by Bessel sequences (we
refer \cite{CHRISTENSEN, HEIL} for Bessel sequences). Balazs and Stoeva  studied
these operators in \cite{BALAZSREP, STOEVABALAZSINVER, STOEVABALAZSINVERONTHE,
STOEVABALAZSRIESZ, STOEVABALAZSCANONICAL, STOEVABALAZSWEIGHTED,
STOEVABALAZSDETAILED}.

Let $\{f_n\}_n$
be a sequence in  the  dual space $\mathcal{X}^*$ of a Banach
space $\mathcal{X}$ and $\{\tau_n\}_n$ be a sequence in a Banach space
$\mathcal{Y}$. The operator $\tau \otimes f$ is defined by  	$\tau
\otimes f:\mathcal{X}\ni x \mapsto f(x)\tau\in \mathcal{Y}$. It was  Rahimi and
Balazs \cite{RAHIMIBALAZSMUL} who extended the operator in (\ref{FIRST
EQUATION}) from Hilbert spaces to Banach spaces. For a Banach space
$\mathcal{X}$, and dual $\mathcal{X}^*$, they considered the operator 
\begin{align}\label{SECOND EQUATION}
\sum_{n=1}^{\infty}\lambda_n (\tau_n\otimes f_n).
\end{align}
Rahimi and Balazs  studied the operator in (\ref{SECOND EQUATION}), whenever
$\{\tau_n\}_n$ p-Bessel sequence (we refer \cite{CASAZZACHRISTENSENSTOEVA, CHRISTENSENSTOEVA} for p-Bessel sequences) for
$\mathcal{X}^*$ and $\{f_n\}_n$ q-Bessel sequence for $\mathcal{X}$ ($q$ is
conjugate index of $p$). Besides theoretical importance, multipliers also play
important role in frame (particularly Gabor) multipliers
\cite{FEICHTINGERNOWAK}, signal processing \cite{MATZHLAWATSCH}, computational
auditory scene analysis \cite{WANGBROWN}, sound synthesis \cite{DEPALLE},
psychoacoustics \cite{BALAZSLABACK}, etc.

 In the present paper, we attempt to study the non-linear version of
operator in (\ref{SECOND EQUATION}). In Section \ref{PRELIMINARIES} we recall
necessary  definitions and results which we use. In Section \ref{MULTIPLIERSSEC} we give
definitions of frames for metric spaces, definition of multiplier for metric spaces and study the properties of multiplier.

 \section{Preliminaries}\label{PRELIMINARIES}
 In Hilbert spaces, a Riesz basis is defined as an image of an orthonormal basis
under an invertible operator \cite{CHRISTENSEN}. In order to define Riesz basis
for Banach spaces,  one has  to look for characterizations not involving inner
product.  Following is one such charaterization.
 \begin{theorem}\cite{CHRISTENSEN}
For a sequence $\{\tau_n\}_{n}$ in a Hilbert space $\mathcal{H}$, the following
are equivalent. 
  \begin{enumerate}[\upshape(i)]
\item $\{\tau_n\}_{n}$ is a Riesz basis for $\mathcal{H}$.
 	\item $\overline{\operatorname{span}}\{\tau_n\}_{n}=\mathcal{H}$ and 
there exist $a,b>0$ such that for every finite subset $\mathbb{S}$ of
$\mathbb{N}$, 
 	\begin{align*}
 	a \left(\sum_{n \in \mathbb{S}}|c_n|^2\right)^\frac{1}{2}\leq
\left\|\sum_{n \in \mathbb{S}}c_n\tau_n\right\|\leq b \left(\sum_{n \in
\mathbb{S}}|c_n|^2\right)^\frac{1}{2}, \quad \forall c_n \in \mathbb{K}.
 	\end{align*}
 \end{enumerate}	
  \end{theorem}
 
  \begin{definition}\cite{ALDROUBISUNTANG}
 Let $1<q<\infty$ and $\mathcal{X}$ be a Banach space. A collection 	$\{\tau_n\}_{n}$  in $\mathcal{X}$ is said to be a 
 \begin{enumerate}[\upshape(i)]
 	\item q-Riesz sequence for $\mathcal{X}$ if there exist $a,b>0$ such that for every finite subset $\mathbb{S}$ of $\mathbb{N}$, 
 	\begin{align}\label{RIESZSEQUENCEINEQUALITY}
 	a \left(\sum_{n \in \mathbb{S}}|c_n|^q\right)^\frac{1}{q}\leq \left\|\sum_{n \in \mathbb{S}}c_n\tau_n\right\|\leq b \left(\sum_{n \in \mathbb{S}}|c_n|^q\right)^\frac{1}{q}, \quad \forall c_n \in \mathbb{K}.
 	\end{align}
 	\item q-Riesz basis for $\mathcal{X}$ if it is a q-Riesz sequence  for
$\mathcal{X}$ and $\overline{\operatorname{span}}\{\tau_n\}_{n}=\mathcal{X}$.
 \end{enumerate}
  \end{definition}
  We now recall the definition of a frame for a Hilbert space.
  \begin{definition}\cite{CHRISTENSEN}\label{FRAMEDEFINITION}
  	A collection $\{\tau_n\}_{n}$ in a Hilbert space $\mathcal{H}$ is said to be a frame for $\mathcal{H}$ if there exist $a,b>0$ such that 
  	\begin{align*}
  	a\|h\|^2\leq \sum_{n=1}^{\infty}|\langle h, \tau_n\rangle|^2\leq b \|h\|^2, \quad \forall h \in \mathcal{H}.
  	\end{align*}
  \end{definition}
  By realizing that the functional $\mathcal{H} \ni h \mapsto \langle h, \tau_n\rangle \in \mathbb{K}$ is bounded linear, Definition \ref{FRAMEDEFINITION} leads to the following in Banach spaces.
  \begin{definition}\cite{ALDROUBISUNTANG, CHRISTENSENSTOEVA}
  	Let $1<p<\infty$ and $\mathcal{X}$ be a Banach space. 
  \begin{enumerate}[\upshape(i)]
  	\item A collection 	$\{f_n\}_{n}$ of bounded linear functionals in $\mathcal{X}^*$ is said to be a p-frame for $\mathcal{X}$ if there exist $a,b>0$ such that 	
  	\begin{align*}
  	a\|x\|\leq \left(\sum_{n=1}^{\infty}|f_n(x)|^p\right)^\frac{1}{p}\leq b\|x\|,\quad \forall x \in \mathcal{X}.
  	\end{align*}
  	\item A collection 	$\{\tau_n\}_{n}$  in $\mathcal{X}$ is said to
be a p-frame for $\mathcal{X}^*$ if there exist $a,b>0$ such that 
  	\begin{align*}
  	a\|f\|\leq \left(\sum_{n=1}^{\infty}|f(\tau_n)|^p\right)^\frac{1}{p}\leq b\|f\|,\quad \forall f \in \mathcal{X}^*.
  	\end{align*}
  \end{enumerate}	
  \end{definition}
 For more  about p-frames for Banach spaces we refer \cite{STOEVAON,
CASAZZACHRISTENSENSTOEVA, STOEVAPERT, STOEVAGEN, STOEVACON, STOEVACONXD}.

 We
now recall the definition of Lipschitz function. Let $\mathcal{M}$, 
$\mathcal{N}$ be  metric spaces. A function $f:\mathcal{M}  \rightarrow
\mathcal{N}$ is said to be Lipschitz if there exists $b> 0$ such that 
  	\begin{align*}
  	d(f(x), f(y)) \leq b\, d(x,y), \quad \forall x, y \in \mathcal{M}.
  	\end{align*}
   \begin{definition}\cite{WEAVER}
   	Let	$\mathcal{X}$ be a Banach space.
   	\begin{enumerate}[\upshape(i)]
   		\item Let $\mathcal{M}$ be a  metric space. The collection 	$\operatorname{Lip}(\mathcal{M}, \mathcal{X})$
   		is defined as $\operatorname{Lip}(\mathcal{M}, \mathcal{X})\coloneqq \{f:f:\mathcal{M} 
   		\rightarrow \mathcal{X}  \operatorname{ is ~ Lipschitz} \}.$ For $f \in \operatorname{Lip}(\mathcal{M}, \mathcal{X})$, the Lipschitz number 
   		is defined as 
   		\begin{align*}
   		\operatorname{Lip}(f)\coloneqq \sup_{x, y \in \mathcal{M}, x\neq
   			y} \frac{\|f(x)-f(y)\|}{d(x,y)}.
   		\end{align*}
   		\item Let $(\mathcal{M}, 0)$ be a pointed metric space. The collection 	$\operatorname{Lip}_0(\mathcal{M}, \mathcal{X})$
   		is defined as $\operatorname{Lip}_0(\mathcal{M}, \mathcal{X})\coloneqq \{f:f:\mathcal{M} 
   		\rightarrow \mathcal{X}  \operatorname{ is ~ Lipschitz ~ and } f(0)=0\}.$
   		For $f \in \operatorname{Lip}_0(\mathcal{M}, \mathcal{X})$, the Lipschitz norm
   		is defined as 
   		\begin{align*}
   		\|f\|_{\operatorname{Lip}_0}\coloneqq \sup_{x, y \in \mathcal{M}, x\neq
   			y} \frac{\|f(x)-f(y)\|}{d(x,y)}.
   		\end{align*}
   	\end{enumerate}

   \end{definition}
   
   \begin{theorem}\cite{WEAVER}
   	Let	$\mathcal{X}$ be a Banach space.
   	\begin{enumerate}[\upshape(i)]
   		\item If $\mathcal{M}$ is a  metric space, then 	$\operatorname{Lip}(\mathcal{M},
   		\mathcal{X})$ is a semi-normed vector  space w.r.t. the semi-norm  $\operatorname{Lip}(\cdot)$.
   		\item If $(\mathcal{M}, 0)$ is a pointed metric space, then 	$\operatorname{Lip}_0(\mathcal{M},
   		\mathcal{X})$ is a Banach space w.r.t. the  norm
   		$\|\cdot\|_{\operatorname{Lip}_0}$.
   	\end{enumerate}
   \end{theorem}	
 The spaces $\operatorname{Lip}(\mathcal{M},
 \mathcal{X})$ and $\operatorname{Lip}_0(\mathcal{M},
 \mathcal{X})$  are well-studied and we refer \cite{WEAVER,
COBZASMICULESCUNICOLAE, PIASECKI, KALTONGODEFROY, KALTON} for further
information.

In the theory of bounded linear operators between Banach spaces, an operator is
said to be compact if the image of the unit ball under the operator is
precompact  \cite{FABIAN}. Linearity of the operator now gives various
 charaterizaions of
compactness and plays important role in rich theories such as theory of integral equations \cite{DEBNATHMIKUSINSKI}, spectral
theory \cite{BEAUZAMY}, theory of Fredholm operators \cite{EDMUNDSEVANS}, operator algebra (C*-algebra) \cite{DAVIDSON},
K-theory \cite{RORDAM}, Calkin algebra \cite{CARADUSPFAFFENBERGERYOOD}, (operator) ideal theory \cite{PIETSCH},  approximation
properties of Banach spaces \cite{JOHNSONLINDENSTRAUSS}, Schauder basis theory  \cite{JOHNSONLINDENSTRAUSS}. Lack of linearity is a
hurdle when one tries to define compactness of non-linear maps. This hurdle was
sucessefully crossed in the paper which began the study of Lipschitz compact
operators. We now record these things which are necessary in the paper.
 \begin{definition}\cite{JIMENEZSEPUILCREMOISES}
 If $\mathcal{M}$ is a  metric space and $\mathcal{X}$ is a Banach space, then the  Lipschitz image of a Lipschitz map (also called as Lipschitz operator) $f:\mathcal{M}\rightarrow \mathcal{X}$ is defined as the set 
 \begin{align}\label{LIPSCHITZIMAGE}
 \left\{\frac{f(x)-f(y)}{d(x,y)}:x, y \in \mathcal{M}, x\neq y\right\}.
 \end{align}
 \end{definition}
We observe that whenever an operator is linear, the set in (\ref{LIPSCHITZIMAGE}) is simply the image
of the unit sphere.
 \begin{definition}\cite{JIMENEZSEPUILCREMOISES}
 	If $(\mathcal{M}, 0)$ is a pointed metric space and $\mathcal{X}$ is a
Banach space, then a Lipschitz map  $f:\mathcal{M}\rightarrow \mathcal{X}$ such
that $f(0)=0$ is said to be Lipschitz compact if its Lipschitz image is
relatively compact  in $\mathcal{X}$, i.e., the closure of the set in 
(\ref{LIPSCHITZIMAGE}) is compact  in $\mathcal{X}$.
 \end{definition}
 As showed in \cite{JIMENEZSEPUILCREMOISES}, there is a large collection of Lipschitz compact operators. To state this, first we need a definition.
 \begin{definition}\cite{CHENZHENG}
 Let $(\mathcal{M}, 0)$ be a pointed metric space and $\mathcal{X}$ be a Banach space. A Lipschitz operator $f:\mathcal{M}\rightarrow \mathcal{X}$ such that $f(0)=0$ is said to be strongly Lipschitz p-nuclear ($1\leq p <\infty$) if there exist operators $A \in \mathcal{B}(\ell^p(\mathbb{N}), \mathcal{X})$, $g \in \operatorname{Lip}_0(\mathcal{M},
 \ell^\infty(\mathbb{N}))$ and a diagonal operator $M_\lambda \in \mathcal{B}(\ell^\infty(\mathbb{N}), \ell^p(\mathbb{N}))$ induced by a sequence $\lambda \in \ell^p(\mathbb{N})$ such that $f=AM_\lambda g$, i.e., the following diagram commutes.
\begin{center}
\[
\begin{tikzcd}
\mathcal{M} \arrow[r, "f"]\arrow[d, "g"]& \mathcal{X} \\
\ell^\infty(\mathbb{N})\arrow[r, "M_\lambda" ]& \ell^p(\mathbb{N})\arrow[u, "A"
]
\end{tikzcd}	
\]
\end{center}

 \end{definition} 
 \begin{proposition}\cite{JIMENEZSEPUILCREMOISES}
 	Every strongly Lipschitz p-nuclear operator from a pointed metric space to a Banach space is Lipschitz compact.
 \end{proposition}
 Since the image of a linear operator is a subspace, the natural definition of finite rank operator is that image is a finite dimensional subspace. The image of Lipschitz map may not be a subspace. Thus care has to be taken while defining rank of such maps.
 \begin{definition}\cite{JIMENEZSEPUILCREMOISES}\label{LFR}
 	If $(\mathcal{M}, 0)$ is a pointed metric space and $\mathcal{X}$ is a
Banach space, then a Lipschitz function $f:\mathcal{M}\rightarrow \mathcal{X}$
such that $f(0)=0$ is said to have Lipschitz finite dimensional rank if the linear
hull of its Lipschitz image is a finite dimensional subspace of $\mathcal{X}$.
 \end{definition}
 \begin{definition}\cite{JIMENEZSEPUILCREMOISES}\label{FR}
 	If $\mathcal{M}$ is a  metric space and $\mathcal{X}$ is a Banach space, then a Lipschitz function $f:\mathcal{M}\rightarrow \mathcal{X}$  is said to have  finite dimensional rank if the linear hull of its  image is a finite dimensional subspace of $\mathcal{X}$.
 \end{definition}
Next theorem shows that for pointed metric spaces, Definitions \ref{LFR} and \ref{FR}   are
equivalent. 
 \begin{theorem}\cite{JIMENEZSEPUILCREMOISES, ACHOUR}\label{LIPSCHITCOMPACTIFFLINEAR}
 	Let  $(\mathcal{M}, 0)$ be a pointed metric space and $\mathcal{X}$ be a Banach space.  For  a Lipschitz function $f:\mathcal{M}\rightarrow \mathcal{X}$ such that $f(0)=0$ the following are equivalent.
 \begin{enumerate}[\upshape(i)]
 	\item $f$ has Lipschitz finite dimensional rank.
 	\item $f$ has finite dimensional rank.
 	\item There exist $ f_1, \dots, f_n$ in  $\operatorname{Lip}_0(\mathcal{M}, \mathbb{K})$ and $\tau_1, \dots, \tau_n$ in $\mathcal{X}$ such that
 	\begin{align*}
 	f(x)=\sum_{k=1}^{n}f_k(x)\tau_k, \quad \forall x \in \mathcal{M}.
 	\end{align*} 
 \end{enumerate}	
 \end{theorem}

In Hilbert spaces (and not in Banach spaces), every compact operator is approximable by finite rank
operators in the operator norm \cite{FABIAN}. Following is the definition of approximable
operator for Lipschitz maps.
  \begin{definition}\cite{JIMENEZSEPUILCREMOISES}
  	If $(\mathcal{M}, 0)$ is a pointed metric space and $\mathcal{X}$ is a Banach space, then a Lipschitz function $f:\mathcal{M}\rightarrow \mathcal{X}$ such that $f(0)=0$ is said to be Lipschitz approximable if it is the limit in the Lipschitz norm of a sequence of Lipschitz finite rank operators from  $\mathcal{M}$ to  $\mathcal{X}$.
  \end{definition}
  \begin{theorem}\cite{JIMENEZSEPUILCREMOISES}\label{LIPSCHITZAPPROMABLEISCOMPACT}
  	Every Lipschitz approximable operator from pointed metric space $(\mathcal{M}, 0)$ to  a Banach space $\mathcal{X}$ is Lipschitz compact.
  \end{theorem}

 \section{Multipliers for  Lipschitz p-Bessel sequences in metric
spaces and its properties}\label{MULTIPLIERSSEC}
 We first define the notion of frames for metric spaces.
 \begin{definition}\label{FIRST}(p-frame for metric space)
 	Let $(\mathcal{M},d)$,  $(\mathcal{N}_n,d_n)$, $1\leq n < \infty$ be  
metric
spaces. A collection $\{f_n\}_{n}$ of Lipschitz functions,  $f_n:\mathcal{M}
\to \mathcal{N}_n$ is said to be a Lipschitz p-frame ($1\leq p <\infty$) for  $\mathcal{M}$ relative to
$\{\mathcal{N}_n\}_n$ if 
 	there exist $a,b>0$ such that 
 	\begin{align*}
 	a\,d(x,y)\leq \left(\sum_{n=1}^{\infty}d_n(f_n(x),
f_n(y))^p\right)^\frac{1}{p}\leq b\,d(x,y),\quad \forall x, y \in \mathcal{M}.
 	\end{align*}
 	If $a$ is allowed to take the value 0, then we say that $\{f_n\}_{n}$  a  Lipschitz p-Bessel sequence for  $\mathcal{M}$.
 \end{definition}
  \begin{definition}\label{SECOND}(p-frame for metric space w.r.t. scalars)
  	Let $\mathcal{M}$ be a metric space. A collection $\{f_n\}_{n}$ of Lipschitz functions from  $\mathcal{M}$ to  $\mathbb{K}$ is said to be a Lipschitz p-frame for  $\mathcal{M}$ if there exist $a,b>0$ such that 
  	\begin{align*}
  	a\,d(x,y)\leq \left(\sum_{n=1}^{\infty}|f_n(x)-f_n(y)|^p\right)^\frac{1}{p}\leq b\,d(x,y),\quad \forall x, y \in \mathcal{M}.
  	\end{align*}
  \end{definition}

  \begin{definition}\label{THIRD}(p-frame for a pointed metric space w.r.t. scalars)
  	Let $(\mathcal{M}, 0)$ be a pointed metric space. A collection $\{f_n\}_{n}$ in $\operatorname{Lip}_0(\mathcal{M}, \mathbb{K})$ is said to be a pointed Lipschitz p-frame for  $\mathcal{M}$ if there exist $a,b>0$ such that 
  	\begin{align*}
  	a\,d(x,y)\leq \left(\sum_{n=1}^{\infty}|f_n(x)-f_n(y)|^p\right)^\frac{1}{p}\leq b\,d(x,y),\quad \forall x, y \in \mathcal{M}.
  	\end{align*}
  	
  \end{definition}
  \begin{definition}\label{LAST}
  	Let $(\mathcal{M}, 0)$ be a pointed metric space. A collection $\{\tau_n\}_{n}$ in $\mathcal{M}$ is said to be a pointed Lipschitz p-frame for  $\operatorname{Lip}_0(\mathcal{M}, \mathbb{K})$ if there exist $a,b>0$ such that 
  	\begin{align*}
  	a\,\|f-g\|_{\operatorname{Lip}}\leq \left(\sum_{n=1}^{\infty}|f(\tau_n)-g(\tau_n)|^p\right)^\frac{1}{p}\leq b\|f-g\|_{\operatorname{Lip}},\quad \forall f, g \in \operatorname{Lip}_0(\mathcal{M}, \mathbb{K}).
  	\end{align*}
  	 \end{definition}
  \begin{remark}
  \begin{enumerate}[\upshape(i)]
  \item Definition \ref{FIRST} even generalizes the notion
of bi-Lipschitz embedding (Ribe program) of metric spaces (we refer
\cite{OSTROVSKII, HEINONEN, SEMMES, NAOR1, NAOR2} for more on bi-Lipschitz
embedding) (in fact, we see this by taking a fixed point $z\in \mathcal{N}$ and
defining $f_n(x)=z, \forall x \in  \mathcal{N}$ and $\forall n >1$). It may
happen that a metric space $\mathcal{M}$ may not embed in another metric space
$\mathcal{N}$ through bi-Lipschitz map. But it may have frames. We give
examples to  illustrate these things after this remark.
  		\item By taking $y=0$ and using $f_n(0)=0$, for all $n \in
\mathbb{N}$, we see from Definition \ref{SECOND} that 
  			\begin{align*}
  			a\,d(x,0)\leq \left(\sum_{n=1}^{\infty}|f_n(x)|^p\right)^\frac{1}{p}\leq b\,d(x,0),\quad \forall x \in \mathcal{M}.
  			\end{align*}
  			In particular, if $\mathcal{M}$ is a Banach space, then 
  			\begin{align*}
  			a\|x\|\leq
\left(\sum_{n=1}^{\infty}|f_n(x)|^p\right)^\frac{1}{p}\leq b\|x\|,\quad \forall
x\in \mathcal{M}.
  			\end{align*}
  			Similarly by taking $g=0$ in Definition \ref{LAST}, we
see that 
  				\begin{align*}
  				a\,\|f\|_{\operatorname{Lip}_0}\leq
\left(\sum_{n=1}^{\infty}|f(\tau_n)|^p\right)^\frac{1}{p}\leq
b\|f\|_{\operatorname{Lip}_0},\quad \forall f \in
\operatorname{Lip}_0(\mathcal{M}, \mathbb{K}).
  				\end{align*}
  		\item If $\mathcal{M}$ is a Banach space and $f_n$'s are all
bounded linear functionals, then Definition \ref{FIRST} becomes (i) in
Definition \ref{FRAMEDEFINITION}.
  	\item Since we only  want the definition of 
Lipschitz p-Bessel sequence, we do not address further properties of Lipschitz
frames for metric spaces  in this paper. However, we make a detailed study of frames for  metric spaces in \cite{MAHESHSAM}.

  \end{enumerate}
  \end{remark}
Let $x, y $ be distinct reals and consider $\mathcal{M}=\{x,y\}$ as a metric
subspace of $\mathbb{R}$. Then $\mathcal{M}$ does not embed in metric spaces 
$\mathcal{N}_1=\{x\}$ or $\mathcal{N}_2=\{y\}$. Define $f_1(x)=f_1(y)=x$ and
$f_2(x)=f_2(y)=y$. Then $f_1$ and $f_2$ are Lipschitz and
$|f_1(x)-f_1(y)|^p+|f_2(x)-f_2(y)|^p=2|x-y|^p$. Hence $\{f_1, f_2\}$ is a
Lipschitz  p-frame for 
$\mathcal{M}$.\\
As another example, consider $m, n \in \mathbb{N}$ with $m<n$. Since a
bi-Lipschitz map is injective and continuous, and there is no continuous
injection from $\mathbb{R}^n$ to $\mathbb{R}^m$ (Corollary 2B.4 in \cite{HATCHER}) it follows that
$\mathbb{R}^n$ cannot be embedded in $\mathbb{R}^m$. Now define
$f_j:\mathbb{R}^n \ni (x_1, \dots, x_n) \mapsto (x_j, x_{j+1}, \dots, x_n, x_1,
\dots ,  x_{m-n+j-1})\in \mathbb{R}^m$ for $1\leq j \leq n$. Then $f_j$ is
Lipschitz for all $1\leq j\leq n$ and 
\begin{align*}
 \sum_{j=1}^n\|f_j(x_1, \dots, x_n)-f_j(y_1, \dots, y_n)\|^p=m\|(x_1, \dots,
x_n)-(y_1, \dots, y_n)\|^p, \\~\forall (x_1, \dots, x_n), (y_1, \dots, y_n) \in
\mathbb{R}^n.
\end{align*}
 Thus  $\{f_1, f_2, \dots, f_n\}$ is a Lipscitz p-frame for
$\mathbb{R}^n$.\\
We next give an example which uses infinite number of Lipschitz functions. Let $1<a<b<\infty.$ Let us take $\mathcal{M}\coloneqq[a,b]$ and define $f_n:\mathcal{M}\to \mathbb{R}$ by 
\begin{align*}
f_0(x)&=1, \quad \forall x \in \mathcal{M}\\
f_n(x)&=\frac{(\log x)^n}{n!}, \quad \forall x \in \mathcal{M}, \forall n\geq1.
\end{align*}
 Then $f_n'(x)=\frac{(\log x)^{(n-1)}}{(n-1)!x}$, $\forall x \in \mathcal{M}, \forall n\geq1.$ Hence $f_n'$ is bounded on $\mathcal{M}$, $\forall n\geq1.$ Proposition 2.2.1 in \cite{COBZASMICULESCUNICOLAE} now tells that $f_n$ is a Lipschitz function, for each $n\geq1.$ For $x, y \in \mathcal{M},$ with $x<y$, we now see that 
 \begin{align*}
 \sum_{n=0}^{\infty}|f_n(x)-f_n(y)|&=\sum_{n=0}^{\infty}\frac{(\log
y)^n}{n!}-\sum_{n=0}^{\infty}\frac{(\log x)^n}{n!}\\
&=e^{\log y}-e^{\log x}=y-x=|x-y|.
 \end{align*}
Hence $\{f_n\}_n$ is a Lipschitz 1-frame for $\mathcal{M}$.

We now set a notation which we use in the paper. Let $\mathcal{M}$ be a  metric
space and $\mathcal{X}$ be a  Banach space. Given $f \in
\operatorname{Lip}(\mathcal{M}, \mathbb{K})$ and $\tau \in \mathcal{X}$,
define 
  	$$\tau\otimes f:\mathcal{M} \ni x \mapsto (\tau\otimes f)(x)\coloneqq f(x)\tau \in \mathcal{X}.$$
  Then it follows that $\tau\otimes f$ is a Lipschitz operator and $\operatorname{Lip}(\tau\otimes f)=\|\tau\|\operatorname{Lip}(f)$.
 In his work \cite{SCHATTEN}, Schatten showed that whenever if $\{\lambda_n\}_n \in
\ell^\infty(\mathbb{N})$ and $\{x_n\}_n$, $\{y_n\}_n$ are orthonormal sequences
in a Hilbert space $\mathcal{H}$, then the map in (\ref{FIRST EQUATION}) is a well-defined bounded linear
operator. In 2007, Balazs showed that even if we
take $\{x_n\}_n$, $\{y_n\}_n$ as Bessel sequences, then also $T$ is
well-defined and bounded. In 2010, Rahimi and Balazs \cite{RAHIMIBALAZSMUL} showed that we can
even define operator $T$ in Banach spaces. More precisely the result is
following.
\begin{theorem}\cite{RAHIMIBALAZSMUL}\label{RAHIMIBALAZS}
 Let $\{f_n\}_{n}$   be a
 p-Bessel sequence  for  a Banach space $\mathcal{X}$ with bound $b$ and
$\{\tau_n\}_{n}$   be a
 q-Bessel sequence  for the dual of  a Banach space $\mathcal{Y}$ with bound
$d$, where $q$ is conjugate index of $p$.  If
$\{\lambda_n\}_n \in \ell^\infty(\mathbb{N})$, then the map
  	
  	\begin{align*}
  	T: \mathcal{X} \ni x \mapsto \sum_{n=1}^{\infty}\lambda_n (\tau_n\otimes
f_n) x \in \mathcal{Y}
  	\end{align*}
  	is a well-defined bounded linear   operator with norm 
 at most $bd\|\{\lambda_n\}_n\|_\infty.$
\end{theorem}
We now derive Theorem \ref{RAHIMIBALAZS} in non-linear sense.
  \begin{theorem}\label{DEFINITIONEXISTENCE}
  Let $\{f_n\}_{n}$ in $\operatorname{Lip}_0(\mathcal{M}, \mathbb{K})$  be a
pointed Lipschitz p-Bessel sequence  for a pointed metric space $(\mathcal{M},
0)$ with bound $b$ and $\{\tau_n\}_{n}$ in a Banach space $\mathcal{X}$  be a
pointed Lipschitz q-Bessel sequence for $\operatorname{Lip}_0(\mathcal{X},
\mathbb{K})$ with bound $d$.  If
$\{\lambda_n\}_n \in \ell^\infty(\mathbb{N})$, then the map
  	
  	\begin{align*}
  	T: \mathcal{M} \ni x \mapsto \sum_{n=1}^{\infty}\lambda_n (\tau_n\otimes
f_n) x \in \mathcal{X}
  	\end{align*}
  	is a well-defined Lipschitz  operator such that $T0=0$ with Lipschitz
norm at most $bd\|\{\lambda_n\}_n\|_\infty.$
  \end{theorem}
  \begin{proof}
  	Let $n , m \in \mathbb{N}$ with $n \leq m$. Then for each $x \in
\mathcal{M}$, using Holder's inequality,

  	\begin{align*}
  	\left\|\sum_{k=n}^{m}\lambda_k(\tau_k\otimes f_k)(x)\right\|&=\left\|\sum_{k=n}^{m}\lambda_k f_k(x)\tau_k\right\|=\sup_{\phi \in \mathcal{X}^*,\|\phi\|\leq 1}\left|\phi\left(\sum_{k=n}^{m}\lambda_k f_k(x)\tau_k\right)\right|\\
  	&=\sup_{\phi \in \mathcal{X}^*,\|\phi\|\leq 1}\left|\sum_{k=n}^{m}\lambda_k f_k(x)\phi(\tau_k)\right|\\
  	&\leq\sup_{\phi \in \mathcal{X}^*,\|\phi\|\leq 1}\sum_{k=n}^{m}|\lambda_k| |f_k(x)||\phi(\tau_k)|\\
  	&\leq \sup_{n \in \mathbb{N}}|\lambda_n|\sup_{\phi \in \mathcal{X}^*,\|\phi\|\leq 1}\sum_{k=n}^{m} |f_k(x)||\phi(\tau_k)|\\
  	&\leq \sup_{n \in \mathbb{N}}|\lambda_n|\sup_{\phi \in \mathcal{X}^*,\|\phi\|\leq 1}\left(\sum_{k=n}^{m}|f_k(x)|^p\right)^\frac{1}{p}\left(\sum_{k=n}^{m}|\phi(\tau_k)|^q\right)^\frac{1}{q}\\
  	&\leq \sup_{n \in \mathbb{N}}|\lambda_n|\sup_{\phi \in
\mathcal{X}^*,\|\phi\|\leq
1}\left(\sum_{k=n}^{m}|f_k(x)|^p\right)^\frac{1}{p}d\|\phi\|\\
  	&=d\sup_{n \in \mathbb{N}}|\lambda_n|\left(\sum_{k=n}^{m}|f_k(x)|^p\right)^\frac{1}{p}.
  	\end{align*}
  	Since $\left(\sum_{k=1}^{\infty}|f_k(x)|^p\right)^\frac{1}{p}$
converges, $\sum_{k=1}^{\infty}\lambda_k(\tau_k\otimes f_k)(x)$ also converges.
Now for all $x,y \in \mathcal{M}$,

  	\begin{align*}
  	\|Tx-Ty\|&=\left\|\sum_{n=1}^{\infty}\lambda_n f_n(x)\tau_n-\sum_{n=1}^{\infty}\lambda_n f_n(y)\tau_n\right\|=\left\|\sum_{n=1}^{\infty}\lambda_n (f_n(x)-f_n(y))\tau_n\right\|\\
  	&=\sup_{\phi \in \mathcal{X}^*,\|\phi\|\leq 1}\left|\phi\left(\sum_{n=1}^{\infty}\lambda_n (f_n(x)-f_n(y))\tau_k\right)\right|\\
  	&=\sup_{\phi \in \mathcal{X}^*,\|\phi\|\leq 1}\left|\sum_{n=1}^{\infty}\lambda_n (f_n(x)-f_n(y))\phi(\tau_k)\right|\\
  	&\leq \sup_{n \in \mathbb{N}}|\lambda_n|\sup_{\phi \in \mathcal{X}^*,\|\phi\|\leq 1}\left(\sum_{n=1}^{\infty}|f_n(x)-f_n(y)|^p\right)^\frac{1}{p}\left(\sum_{n=1}^{\infty}|\phi(\tau_n)|^q\right)^\frac{1}{q}\\
  	&\leq \sup_{n \in \mathbb{N}}|\lambda_n|\sup_{\phi \in \mathcal{X}^*,\|\phi\|\leq 1}\left(\sum_{n=1}^{\infty}|f_n(x)-f_n(y)|^p\right)^\frac{1}{p}d\|\phi\|\\
  	&=d\sup_{n \in
\mathbb{N}}|\lambda_n|\left(\sum_{n=1}^{\infty}|f_n(x)-f_n(y)|^p\right)^\frac{1}
{p}\leq bd\sup_{n \in \mathbb{N}}|\lambda_n|d(x,y).
  	\end{align*}
  
  Hence 
  \begin{align*}
  \|T\|_{\operatorname{Lip}_0}=\sup_{x, y \in \mathcal{M}, x\neq y}
\frac{\|Tx-Ty\|}{d(x,y)}\leq bd\sup_{n \in \mathbb{N}}|\lambda_n|.
  \end{align*}
  \end{proof}
\begin{corollary}
  Let $\{f_n\}_{n}$ in $\operatorname{Lip}(\mathcal{M}, \mathbb{K})$  be a
 Lipschitz p-Bessel sequence  for  a metric space $\mathcal{M}$ with
bound $b$ and $\{\tau_n\}_{n}$ in a Banach space $\mathcal{X}$  be a
pointed Lipschitz q-Bessel sequence for $\operatorname{Lip}_0(\mathcal{X},
\mathbb{K})$ with bound $d$.  If
$\{\lambda_n\}_n \in \ell^\infty(\mathbb{N})$, then for fixed $z \in
\mathcal{M}$, the map
  	
  	\begin{align*}
  	T: \mathcal{M} \ni x \mapsto \sum_{n=1}^{\infty}\lambda_n (\tau_n\otimes
(f_n-f(z)) )x \in \mathcal{X}
  	\end{align*}
  	is a well-defined Lipschitz  operator  with Lipschitz
number at most $bd\|\{\lambda_n\}_n\|_\infty.$
\end{corollary}
\begin{proof}
 Define $g_n\coloneqq f_n-f(z), \forall n \in \mathbb{N}$. Then for all $x, y 
\in \mathcal{M}$,
$\left(\sum_{n=1}^{\infty}|g_n(x)-g_n(y)|^p\right)^\frac{1}{p}=\left(\sum_{n=1}^
{\infty}|f_n(x)-f_n(y)|^p\right)^\frac{1}{p}\leq b\,d(x,y)$. Hence $\{g_n\}_{n}$
is a 
 Lipschitz p-Bessel sequence  for  pointed metric space $(\mathcal{M},
z)$ and we apply Theorem \ref{DEFINITIONEXISTENCE} to $\{g_n\}_{n}$.
\end{proof}

   \begin{definition}\label{DEFINITION}
   	Let $\{f_n\}_{n}$ in $\operatorname{Lip}_0(\mathcal{M}, \mathbb{K})$  be
a pointed Lipschitz p-Bessel sequence  for a pointed metric space
$(\mathcal{M},
0)$ and $\{\tau_n\}_{n}$ in a Banach space $\mathcal{X}$  be a pointed Lipschitz
q-Bessel sequence for $\operatorname{Lip}_0(\mathcal{X}, \mathbb{K})$.  Let 
$\{\lambda_n\}_n \in \ell^\infty(\mathbb{N})$. The Lipschitz operator
   	\begin{align*}
   	M_{\lambda,f, \tau}\coloneqq  \sum_{n=1}^{\infty}\lambda_n  (\tau_n\otimes f_n)
   	\end{align*}
   	is called as the Lipschitz $(p,q)$-Bessel multiplier. The sequence $\{\lambda_n\}_n$ is called as symbol for $	M_{\lambda,f, \tau}.$
   	\end{definition} 
   We easily see that 	Definition \ref{DEFINITION} generalizes Definition 3.2
in \cite{RAHIMIBALAZSMUL}. By varying the symbol and fixing other parameters in
the multiplier we get map from $\ell^\infty(\mathbb{N})$ to  $\operatorname{Lip}_0(\mathcal{M}, \mathcal{X})$. Property of
this map for Hilbert space was derived by Balazs (Lemma in \cite{BALAZSBASIC})
and for Banach spaces it is due to Rahimi and Balazs (Proposition 3.3 in 
\cite{RAHIMIBALAZSMUL}). In the next proposition we study it in the context of
metric spaces.
  \begin{proposition}
  	Let $\{f_n\}_{n}$ in $\operatorname{Lip}_0(\mathcal{M}, \mathbb{K})$  be
a pointed Lipschitz p-Bessel sequence  for  $(\mathcal{M},0)$  with non-zero
elements, $\{\tau_n\}_{n}$ in $\mathcal{X}$  be a  q-Riesz sequence  for 
$\operatorname{Lip}_0(\mathcal{X}, \mathbb{K})$ and $\{\lambda_n\}_n \in
\ell^\infty(\mathbb{N})$. Then the mapping 
  	\begin{align*}
  	T:\ell^\infty(\mathbb{N})\ni \{\lambda_n\}_n \mapsto M_{\lambda,f, \tau}
\in \operatorname{Lip}_0(\mathcal{M}, \mathcal{X})
  	\end{align*}
  	is a well-defined injective bounded linear operator.	
   \end{proposition}
  \begin{proof}
  From the norm estimate of $M_{\lambda,f, \tau}$, we see that  $T$ is a 
well-defined bounded linear operator. Let $\{\lambda_n\}_n, \{\mu_n\}_n \in
\ell^\infty(\mathbb{N})$ be such that $M_{\lambda,f, \tau}
=T\{\lambda_n\}_n=T\{\mu_n\}_n=M_{\mu,f, \tau} $. Then
$\sum_{n=1}^{\infty}\lambda_n  f_n (x)\tau_n =M_{\lambda,f, \tau}x =M_{\mu,f,
\tau}x=\sum_{n=1}^{\infty}\mu_n  f_n (x)\tau_n $, $\forall x \in
\mathcal{M}$ $\Rightarrow$ $\sum_{n=1}^{\infty}(\lambda_n-\mu_n)  f_n
(x)\tau_n=0$,  $\forall x \in \mathcal{M}$. Now using Inequality
(\ref{RIESZSEQUENCEINEQUALITY}), 
  	\begin{align*}
  	&a \left(\sum_{n=1 }^\infty|(\lambda_n-\mu_n)  f_n
(x)|^q\right)^\frac{1}{q}\leq \left\|\sum_{n=1}^\infty (\lambda_n-\mu_n)  f_n
(x)\tau_n\right\|=0, \quad \forall x \in \mathcal{M}\\
  &\implies (\lambda_n-\mu_n)  f_n (x)=0, \quad \forall n \in \mathbb{N}, 
\forall x \in \mathcal{M}.
  	\end{align*}
  	Let  $n \in \mathbb{N}$ be fixed. Since $f_n\neq 0$, there exists $x \in
\mathcal{M}$ such that $f_n(x)\neq 0$. Therefore we get $\lambda_n-\mu_n=0$. By
varying $n \in \mathbb{N}$ we arrive at $\lambda_n=\mu_n$, $\forall n \in
\mathbb{N}$. Hence $T$ is injective.
  \end{proof}
 The result that a norm-limit of finite rank linear operators (between Banach spaces)  is a compact operator \cite{FABIAN} was generalized to Lipschitz operators in  Theorem \ref{LIPSCHITZAPPROMABLEISCOMPACT} by Jimenez-Vargas, Sepulcre, and Villegas-Vallecillos \cite{JIMENEZSEPUILCREMOISES}. Using Theorem \ref{LIPSCHITZAPPROMABLEISCOMPACT} we can generalize Lemma 3.6 in \cite{RAHIMIBALAZSMUL}.
  \begin{proposition}
  	Let $\{f_n\}_{n}$ in $\operatorname{Lip}_0(\mathcal{M}, \mathbb{K})$  be
a pointed Lipschitz p-Bessel sequence  for  $(\mathcal{M}, 0)$ with bound $b$
and $\{\tau_n\}_{n}$ in $\mathcal{X}$  be a pointed Lipschitz q-Bessel sequence 
for $\operatorname{Lip}_0(\mathcal{X}, \mathbb{K})$ with bound $d$. 	If
$\{\lambda_n\}_n \in c_0(\mathbb{N})$, then $M_{\lambda,f, \tau}$ is a Lipschitz
compact operator.
  \end{proposition}
  \begin{proof}
  	For each  $m \in \mathbb{N}$, define $	M_{\lambda_m,f, \tau}\coloneqq   \sum_{n=1}^{m}\lambda_n  (\tau_n\otimes f_n )$. Then $	M_{\lambda_m,f, \tau}$ is a Lipschitz finite rank operator (from Theorem \ref{LIPSCHITCOMPACTIFFLINEAR}). Now 
  	\begin{align*}
  	 \|M_{\lambda_m,f, \tau}-M_{\lambda,f,
\tau}\|_{\operatorname{Lip}_0}&=\sup_{x, y \in \mathcal{M}, x\neq y}
\frac{\|(M_{\lambda_m,f, \tau}-M_{\lambda,f, \tau})x-(M_{\lambda_m,f,
\tau}-M_{\lambda,f, \tau})y\|}{d(x,y)}\\
  	 &=\sup_{x, y \in \mathcal{M}, x\neq y}
\frac{\left\|\sum_{n=m+1}^{\infty}\lambda_n  f_n
(x)\tau_n-\sum_{n=m+1}^{\infty}\lambda_n  f_n (y)\tau_n\right\|}{d(x,y)}\\
  	 &=\sup_{x, y \in \mathcal{M}, x\neq y}
\frac{\left\|\sum_{n=m+1}^{\infty}\lambda_n  (f_n
(x)-f_n(y))\tau_n\right\|}{d(x,y)}\\
  	 &\leq bd\sup_{m+1\leq n<\infty }|\lambda_n| \to 0 \text{ as } m \to \infty.
  	\end{align*}
  	Hence $M_{\lambda,f, \tau}$ is the limit of a sequence
of Lipschitz finite rank operators $\{M_{\lambda_m,f, \tau}\}_{m=1}^\infty$ with respect to the Lipschitz norm. Thus $M_{\lambda,f, \tau}$ is Lipschitz approximable and from Theorem
\ref{LIPSCHITZAPPROMABLEISCOMPACT} it follows that $M_{\lambda,f, \tau}$ is
Lipschitz compact.
\end{proof}
   We now study the properties of multiplier by changing its parameters.
These are known as continuity properties of multipliers in the literature.
Following result extends Theorem 5.1 in \cite{RAHIMIBALAZSMUL}.
  \begin{theorem}
  Let $\{f_n\}_{n}$ in $\operatorname{Lip}_0(\mathcal{M}, \mathbb{K})$  be a
pointed Lipschitz p-Bessel sequence  for  $\mathcal{M}$ with bound $b$ and
$\{\tau_n\}_{n}$ in $\mathcal{X}$  be a pointed Lipschitz q-Bessel sequence  for
 $\operatorname{Lip}_0(\mathcal{X}, \mathbb{K})$ with bound $d$ and
$\{\lambda_n\}_n \in \ell^\infty(\mathbb{N})$. 	Let $k \in \mathbb{N}$ and let 
$\lambda^{(k)}=\{\lambda_1^{(k)},\lambda_2^{(k)}, \dots \}$,
$\lambda=\{\lambda_1,\lambda_2, \dots \}$, 
  $\tau^{(k)}=\{\tau_1^{(k)}, \tau_2^{(k)}, \dots\}$,
$\tau_n^{k} \in \mathcal{X}$, $\tau=\{\tau_1, \tau_2, \dots\}$. Assume that for
each $k$, $\lambda^{(k)}\in \ell^\infty(\mathbb{N})$  and
$\tau^{(k)}$  is  a pointed Lipschitz q-Bessel sequence  for
 $\operatorname{Lip}_0(\mathcal{X}, \mathbb{K})$.
  	\begin{enumerate}[\upshape(i)]
  		\item If $\lambda^{(k)} \to \lambda $ as $k \rightarrow \infty $ in p-norm,  then
  		\begin{align*}
  		\|M_{\lambda^{(k)},f, \tau}-M_{\lambda,f, \tau}\|_{\operatorname{Lip}_0} \to 0 \text{ as } k \to \infty.
  		\end{align*}
 \item If $\{\lambda_n\}_n \in \ell^p(\mathbb{N})$ and  $\sum_{n=1}^{\infty}\|\tau_n^{(k)}-\tau_n\|^q \to 0 \text{ as } k \to \infty $, then 
  		 \begin{align*}
  		 \|M_{\lambda, f, \tau^{(k)}}-M_{\lambda,f, \tau}\|_{\operatorname{Lip}_0} \to 0 \text{ as } k \to \infty.
  		 \end{align*}
 \end{enumerate}
  \end{theorem}
  \begin{proof}
  \begin{enumerate}[\upshape(i)]
  		\item Using Theorem \ref{DEFINITIONEXISTENCE},
  		\begin{align*}
  		&\|M_{\lambda^{(k)},f, \tau}-M_{\lambda,f,
\tau}\|_{\operatorname{Lip}_0}\\&=\sup_{x, y \in \mathcal{M}, x\neq y}
\frac{\|(M_{\lambda^{(k)},f, \tau}-M_{\lambda,f, \tau})x-(M_{\lambda^{(k)},f,
\tau}-M_{\lambda,f, \tau})y\|}{d(x,y)}\\
  		&=\sup_{x, y \in \mathcal{M}, x\neq y}
\frac{\left\|\sum_{n=1}^{\infty}(\lambda_n^{(k)}-\lambda_n)f_n(x)\tau_n-\sum_{
n=1}^{\infty}(\lambda_n^{(k)}-\lambda_n)f_n(y)\tau_n\right\|}{d(x,y)}\\
  		&=\sup_{x, y \in \mathcal{M}, x\neq y}
\frac{\left\|\sum_{n=1}^{\infty}(\lambda_n^{(k)}
-\lambda_n)(f_n(x)-f_n(y))\tau_n\right\|}{d(x,y)}\\
  	  &\leq bd \sup_{n\in \mathbb{N}}|\lambda_n^{(k)}-\lambda_n|=bd \|\{\lambda_n^{(k)}-\lambda_n\}_n\|_\infty \\
  		& \leq bd \|\{\lambda_n^{(k)}-\lambda_n\}_n\|_p \to 0 \text{ as } k \to \infty.
  		\end{align*}
  		\item Using Holder's inequality,
  			\begin{align*}
  		&	\|M_{\lambda,f, \tau^{(k)}}-M_{\lambda,f, \tau}\|_{\operatorname{Lip}_0}\\
  		&=\sup_{x, y \in \mathcal{M}, x\neq y} \frac{\|(M_{\lambda,f,
\tau^{(k)}}-M_{\lambda,f, \tau})x-(M_{\lambda,f, \tau^{(k)}}-M_{\lambda,f,
\tau})y\|}{d(x,y)}\\
  			&=\sup_{x, y \in \mathcal{M}, x\neq y}
\frac{\left\|\sum_{n=1}^{\infty}\lambda_nf_n(x)(\tau_n^{(k)}-\tau_n)-\sum_{n=1}^
{\infty}\lambda_nf_n(y)(\tau_n^{(k)}-\tau_n)\right\|}{d(x,y)}\\
  			&=\sup_{x, y \in \mathcal{M}, x\neq y}
\frac{\left\|\sum_{n=1}^{\infty}\lambda_n(f_n(x)-f_n(y))(\tau_n^{(k)}
-\tau_n)\right\|}{d(x,y)}\\
  			&=\sup_{x, y \in \mathcal{M}, x\neq y} \ \ \sup _{\phi \in
\mathcal{X}^*,\|\phi\|\leq 1}
\frac{\left|\sum_{n=1}^{\infty}\lambda_n(f_n(x)-f_n(y))\phi(\tau_n^{(k)}
-\tau_n)\right|}{d(x,y)}\\
  			&\leq \sup_{x, y \in \mathcal{M}, x\neq y} \ \ \sup _{\phi
\in \mathcal{X}^*,\|\phi\|\leq 1}
\frac{\left(\sum_{n=1}^{\infty}|\lambda_n(f_n(x)-f_n(y))|^p\right)^\frac{1}{p}
\left(\sum_{n=1}^{\infty}|\phi(\tau_n^{(k)}-\tau_n)|^q\right)^\frac{1}{q}}{d(x,
y)}\\
  			&\leq \sup_{x, y \in \mathcal{M}, x\neq y} \ \ \sup _{\phi
\in \mathcal{X}^*,\|\phi\|\leq 1}
\frac{\left(\sum_{n=1}^{\infty}|\lambda_n|^p\right)^\frac{1}{p}\left(\sum_{n=1}^
{\infty}|f_n(x)-f_n(y)|^p\right)^\frac{1}{p}\left(\sum_{n=1}^{\infty}
|\phi(\tau_n^{(k)}-\tau_n)|^q\right)^\frac{1}{q}}{d(x,y)}\\
  			&\leq b \|\{\lambda_n\}_n\|_p\left(\sum_{n=1}^{\infty}\|\tau_n^{(k)}-\tau_n\|^q\right)^\frac{1}{q} \to 0 \text{ as } k \to \infty.
  			\end{align*}
  \end{enumerate}	
  \end{proof}

  \section{Acknowledgments}
  The first author thanks the National Institute of Technology Karnataka (NITK), Surathkal for giving him financial support and the present work of the second author was partially supported by National Board for Higher Mathematics (NBHM), Ministry of Atomic Energy, Government of India (Reference
No.2/48(16)/2012/NBHM(R.P.)/R\&D 11/9133).

 \bibliographystyle{plain}

 \bibliography{reference.bib}

\end{document}